\documentclass[a4paper,12pt,reqno]{amsart}
\usepackage[active]{srcltx}
\usepackage{verbatim}
\usepackage{epsfig,graphicx,color,mathrsfs}
\usepackage{graphicx}
\usepackage{amssymb}
\usepackage[english]{babel}
\usepackage[T1]{fontenc}
\usepackage[latin9]{inputenc}
\usepackage{amsmath}
\usepackage{amsthm}
\usepackage{amssymb}
\usepackage{xcolor}

\makeatletter
\theoremstyle{plain}
\newtheorem{thm}{\protect\theoremname}
\theoremstyle{plain}
\newtheorem{defn}[thm]{\protect\definitionname}
\theoremstyle{plain}
\newtheorem{prop}[thm]{\protect\propositionname}
\theoremstyle{plain}
\newtheorem{lem}[thm]{\protect\lemmaname}
\theoremstyle{plain}

\theoremstyle{plain}
\newtheorem{rem}[thm]{\protect\Remarkname}
\theoremstyle{plain}

\makeatother

\usepackage{babel}
\providecommand{\lemmaname}{Lemma}
\providecommand{\definitionname}{Definition}
\providecommand{\propositionname}{Proposition}
\providecommand{\theoremname}{Theorem}
\providecommand{\Remarkname}{Remark}
\providecommand{\conjecturename}{Conjecture}
\providecommand{\corollaryname}{Corollary}

\newcommand{\Rn}{\ensuremath  \mathbb{R}^N_+}
\newcommand{\blambda}{\ensuremath \bar{\lambda}}

\newcommand{\vep}{\varepsilon}

\newcommand{\be}{\begin{equation}}
	\newcommand{\ee}{\end{equation}}

\newcommand{\tu}{\tilde{u}}

\thanks{$^*$ corresponding author}

\title[Monotonicity results in half spaces]{Monotonicity results in half spaces for quasilinear elliptic equations involving a singular term}

\email{luigi.montoro@unical.it., luigi.muglia@unical.it, }\email{berardino.sciunzi@unical.it}

\address{Department of Mathematics and Computer Science, UNICAL, Rende (CS), Italy}
\author[L.\ Montoro]{Luigi Montoro}
\author[L.\ Muglia]{Luigi Muglia}
\author[B.\ Sciunzi]{Berardino Sciunzi $^*$}

\begin{document}

\begin{abstract}
We consider positive solutions to $\displaystyle -\Delta_p u=\frac{1}{u^\gamma}+f(u)$ under zero Dirichlet condition  in the half space. Exploiting a prio-ri estimates and the moving plane technique, we prove that  any solution is monotone increasing in the direction orthogonal to the boundary. %\emph{We are truly very grateful to swollyman for showing us the light.}
\end{abstract}
%\MSC[2020]{35J75, 35A02, 35B09}

\maketitle

\section{Introduction}
We deal with the study of monotonicity properties of positive weak solutions to quasilinear elliptic problems involving singular nonlinearities,  in the half space.
In particular we consider the problem:
\begin{equation}\tag{$\mathcal P_{\gamma}$}\label{MP}
\begin{cases}
\displaystyle -\Delta_p u=\frac{1}{u^\gamma}+f(u) & \text{in} \,\, \mathbb{R}^N_+\\
u>0 & \text{in} \,\,  \mathbb{R}^N_+\\
u=0 & \text{on} \,\, \partial \mathbb{R}^N_+,
\end{cases}
\end{equation}
where $N\geq 2$, $\gamma >1$ and $\mathbb{R}^N_+:=\{x\in \mathbb{R}^N:x_N>0\}$. It is well known that solutions of $p$-Laplace equations are generally of class $C^{1,\alpha}$ (see \cite{Di,Li,T}) and the equation has to be understood in the weak sense. The presence of the singular term $\displaystyle {u^{-\gamma}}$ also causes that solutions are not smooth up to the boundary and, more important, are not in suitable energy spaces near the boundary. There is a wide literature regarding this challenging issue, see e.g. \cite{A,Bo,BM,Br,GrSc, CraRaTa, DP,FMSL,LM, nostro2, oliva,OlPet} and the references therein.  

Monotonicity properties of the solutions represent a crucial information that brings advantage 
 in many applications: blow-up analysis, a-priori estimates and also in the proofs of Liouville type theorems. The study of the monotonicity of the solutions was started in the semilinear nondegenerate case in a series of papers.  We refer to \cite{BCN1,BCN2,BCN3} and to \cite{Dan1,Dan2}.  The more general case of domains that are epigraphs has been recently solved in \cite{BFS}.  The study of the quasilinear case with smooth nonlinearities was started and settled in   \cite{FMS, FMS MathAnn, FMRS, FMS PISA}. Here we address the case in which the nonlinearity has a singular part. \\

\noindent In the sequel we will denote by
$$
 \Sigma_{\lambda}:=\{x\in \mathbb{R}^N_+:0<x_N<{\lambda}\},
$$
and, for a given $x\in\mathbb{R}^N_+$ we will use the notation $x=(x',x_N)$ where $x'\in\mathbb{R}^{n-1}$.

About the solution $u$ and the nonlinearity  $f:[0,+\infty)\to \mathbb{R}$ we will suppose the following: \\\begin{itemize}
 \item[{\bf (hp):} (i)] $\displaystyle \frac{1}{t^\gamma}+f(t)>0$ whenever $t>0$.
%\item[(ii)] There exists $t_0>0$ such that $t^{\gamma}f(t)$ is bounded on $[0,t_0]$.
\item[(ii)] $f:[0,+\infty)\to \mathbb{R}$ is locally Lipschitz continuous.  
\item[(iii)] The solution $u$ is locally bounded in  strips, namely  for all ${\lambda}>0$ there exists $\tilde{C}:=\tilde C(\lambda)$ such that 
$$
 0\leq u(x)\leq \tilde C, \qquad x\in \Sigma_{\lambda}.
$$
\end{itemize}
\ \\
\noindent As recalled above   the solutions, in general, have not $W^{1,p}$ regularity up to the boundary (see \cite{LM}) and the equation has to be understood in the following distributional way:

\begin{defn}We say that $u$ is a solution to \eqref{MP} if and only if  $u\in W^{1,p}_{loc}(\mathbb{R}^N_+)$, $\displaystyle {\rm{ess inf}}_K\, u>0$
	for any compact $K\subset \mathbb R^N_+$ and 
\begin{equation}\label{w-sol}
	\int_{\Rn}|\nabla u|^{p-2}(\nabla u,\nabla \varphi)=\int_{\Rn}\left(\frac{1}{u^\gamma}+f(u)\right)\varphi \qquad \forall \varphi\in C^1_c(\mathbb{R}_+^N).
\end{equation}
Moreover, the zero Dirichlet boundary condition also has to be understood in the weak meaning
\[
(u-\varepsilon)^+\varphi\chi_{\mathbb{R}^N_+}\in W^{1,p}_0(\mathbb{R}^N_+)\qquad\forall\,\varphi\in C^1_c(\mathbb{R}^N)\,.
\]
\end{defn}
\begin{rem}
In some cases, see e.g. \cite{Bo} for the case of bounded domains, when dealing with singular equations, it is natural to understand the zero Dirichlet datum requiring that $u^{\frac{\gamma+p-1}{p}}\in W^{1,p}_0(\Omega)$. We stress the fact that such assumption already implies that $(u-\varepsilon)^+\in W^{1,p}_0(\Omega)$ so that our analysis do apply also to this case. 
\end{rem}
 We are now ready to state our main result:

\begin{thm}\label{mainthmczz}
 	Let   $u$  be a solution  to \eqref{MP}. Then, assuming  ${\bf (hp)}$, it follows that 
  $u$ is monotone increasing w.r.t. the $x_N$-direction, that is
 	\[
 	\frac{\partial u}{\partial x_N}\,\geq\,0 \quad \text{in}\quad \mathbb{R}^N_+\,.
 	\]
 \end{thm}
 
 The paper is organized as follows: in Section 2 we carry out the study of the behaviour and the monotonicity of the solutions near the boundary; this is nothing but a barrier argument. In Section 3 we prove Theorem \ref{mainthmczz}. Here we exploit the a refined version of the moving plane technique, taking care of the lack of compactness. 
 
\section{Behaviour and monotonicity of the solutions near to the boundary}

The arguments of this section are well understood in the semilinear case and go back to \cite{DP,GUI,LM}. Our proofs are self contained and, as the reader will guess, exploit some typical barriers arguments. 

Let us recall that here, for a given $x\in\mathbb{R}^N_+$, we will use the notation $x=(x',x_N)$ where $x'\in\mathbb{R}^{n-1}$.  Moreover we will use the following well known inequality (see \cite{Da}),
\begin{equation}\label{eq:lucio}
(|\eta|^{p-2}\eta-|\eta'|^{p-2}\eta',\eta - \eta') \geq C_p (|\eta|+|\eta'|)^{p-2}|\eta-\eta'|^2,
\end{equation}
for every $\eta,\eta' \in \mathbb{R}^N\setminus\{0\}$.\ \\ \ \\
We begin studying the monotonicity of our solutions near the boundary.
The desired estimates shall follow 
studying the behavior of the solutions to the auxiliary problem:
\begin{equation}\label{MPC}
\begin{cases}
\displaystyle -\Delta_p u=\frac{C}{u^\gamma} & \text{in} \,\, \mathbb{R}^N_+,\\
u>0 & \text{in} \,\,  \mathbb{R}^N_+,\\
u=0 & \text{on} \,\, \partial \mathbb{R}^N_+,
\end{cases}
\end{equation}
for a suitable choice of $C$ (due to the assumptions {\bf (hp)}).

In the next lemma, we establish bounds for supersolutions and subsolutions of problems \eqref{MPC}. As a rule, for solutions to \eqref{MPC}, we obtain continuity up to the boundary. Before proceeding with this, and for the sake of completeness, we recall the following:
\begin{defn} We say that
$u$ is a weak supersolution to \eqref{MPC} if   $u\in W^{1,p}_{loc}(\mathbb{R}^N_+)$, ${\rm{ess inf}}_K\, u>0$
	for any compact $K\subset \mathbb R^N_+$ and 
\begin{equation}\label{w-supersol}
	\int_{\Rn}|\nabla u|^{p-2}(\nabla u,\nabla \varphi)\geq C\int_{\Rn}\frac{\varphi}{u^\gamma}\qquad \forall \varphi\in C^1_c(\mathbb{R}_+^N),  \varphi\geq 0.
\end{equation}
We say that $u$ is a weak subsolution to \eqref{MPC} if $u\in W^{1,p}_{loc}(\mathbb{R}^N_+)$, ${\rm{ess inf}}_K\, u>0$
	for any compact $K\subset \mathbb R^N_+$ and 
\begin{equation}\label{w-subsol}
	\int_{\Rn}|\nabla u|^{p-2}(\nabla u,\nabla \varphi)\leq C\int_{\Rn}\frac{\varphi}{u^\gamma}\qquad \forall \varphi\in C^1_c(\mathbb{R}_+^N),  \varphi\geq 0.
\end{equation}
\end{defn}

\begin{lem}\label{continuity}
Let $u\in W^{1,p}_{loc}(\mathbb{R}^N_+)$ be a supersolution to the problem \eqref{MPC}.
Then there exists $c_s>0$ such that:
 \begin{equation}\label{controllo1}
  u(x)\geq c_s x_N^\beta,\qquad \displaystyle \beta=\frac{p}{\gamma+p-1}\,.
 \end{equation}
If  $u\in W^{1,p}_{loc}(\mathbb{R}^N_+)$ is a subsolution to the problem \eqref{MPC} then there exists $\bar{\lambda}>0$  and $C_S>0$ such that
\begin{equation}\label{controllo2}
  u(x)\leq C_Sx_N^\beta.
 \end{equation}
 for all $x\in\Sigma_{\bar{\lambda}}$.\\
%if $\gamma\in(0,1]$ 
%\begin{equation}\label{controllo2}
%  c_s x_N\leq u(x)\leq C_Sx_N^\beta.
% \end{equation}
%\noindent  if $\gamma>1$, 
% \begin{equation}\label{controllo}
%  c_s x_N^\beta\leq u(x)\leq C_Sx_N^\beta;
% \end{equation}
%if $\gamma\in(0,1]$ 
%\begin{equation}\label{controllo2}
%  c_s x_N\leq u(x)\leq C_Sx_N^\beta.
% \end{equation}
Moreover as consequence,  if  $u$ is a solution to \eqref{MPC}, $u\in C^0(\overline{\mathbb{R}^N_+})$.
\end{lem}
\begin{rem}
Let us point out that thanks to Lemma \ref{continuity}, by standard regularity theory, recalling that the solutions are locally bounded and locally bounded away from zero in the interior of the half space,  we deduce that $u\in C^{1,\alpha}_{loc}(\mathbb{R}^N_+)$.
\end{rem}
\begin{proof}
\noindent \textit{We first prove  \eqref{controllo1}}. \\

\noindent Let  $\phi_1$ be the first positive eigenfunction of the $p$-Laplacian operator with zero Dirichlet boundary data on the unit ball $B_1(0)$, corresponding to the first eigenvalue $\lambda_1$.  Recall that $\phi_1\in C^{1}(\overline{B_1(0)})$ and it is radially symmetric.  
Define
$$
 w_\beta:=\tilde{C}\phi_1^\beta, \qquad \tilde{C}>0.
$$
Denoting by
\begin{equation}\label{eta}
\eta(x):=\frac{\tilde{C}^{p-1+\gamma}}{C}\beta^{p-1}\left[\lambda_1\phi_1^p+(1-\beta)(p-1)|\nabla \phi_1|^p\right],
\end{equation}
it is straightforward  to verify that
$$
  -\Delta_pw_\beta=\frac{C\eta(x)}{w_\beta^\gamma}, \qquad \mbox{ on }B_1(0),
$$
where we implicitly take into account that, by Hopf Lemma, $\nabla \phi_1\neq 0$ on the boundary of $B_1(0)$.
Choosing $\tilde{C}$ small enough,  we can assume $\eta(x)<1$ hence
$$
 -\Delta_pw_\beta<\frac{C}{w_\beta^\gamma}, \qquad \mbox{ on }B_1(0).
$$
Let us consider $x_0=(x_0',x_{0,N})\in \mathbb{R}_+^N$ and let 
$$
 w_{\beta,x_0}=x_{0,N}^\beta w_{\beta}\left(\frac{x-x_0}{x_{0,N}}\right),
$$
for all $x\in B_{x_{0,N}}(x_0)$. Thus $w_{\beta,x_0}$ weakly solves
$$
  -\Delta_pw_{\beta,x_0}\leq \frac{C}{w_{\beta,x_0}^\gamma}.
$$
For a given $u$ supersolution of  \eqref{MPC} and $\varepsilon>0$,  let $\varphi=(w_{\beta,x_0}-u-\varepsilon)^+$; this is a suitable test function since $supp(w_{\beta,x_0}-\varepsilon)^+\subset\subset B_{x_{0,N}}(x_0)$ and $u\geq w_{\beta,x_0}$ on $\partial B_{x_{0,N}}$.  In this way $supp(w_{\beta,x_0}-u-\varepsilon)^+\subset\subset B_{x_{0,N}}(x_0)$.
Then,  exploiting also \eqref{eq:lucio}, we deduce that:
\begin{eqnarray*}
	&&C_p\int_{B_{x_{0,N}}}(|\nabla w_{\beta,x_0}|+|\nabla u|)^{p-2}|\nabla (w_{\beta,x_0}-u-\varepsilon)^+|^2\\
	&\leq&\int_{B_{x_{0,N}}}(|\nabla w_{\beta,x_0}|^{p-2}\nabla w_{\beta,x_0}-|\nabla u|^{p-2}\nabla u,\nabla(w_{\beta,x_0}-u-\varepsilon)^+)\\
	&=&C \int_{B_{x_{0,N}}}\left(\frac{1}{w_{\beta,x_0}^\gamma}-\frac{1}{u^\gamma}\right) (w_{\beta,x_0}-u-\varepsilon)^+\leq 0\,,
\end{eqnarray*}
hence
\[
\int_{B_{x_{0,N}}}(|\nabla w_{\beta,x_0}|+|\nabla w_k|)^{p-2}|\nabla (w_{\beta,x_0}-u-\varepsilon)^+|^2=0.
\]
This permits to get that
\begin{equation}\label{p1}
 w_{\beta,x_0}\leq u+\varepsilon, \qquad \mbox{ in }B_{x_{0,N}}(x_0),  \quad \forall \varepsilon>0,
\end{equation}
then
$$
 x_{0,N}^\beta \tilde{C}\phi_1^{\beta}\left(\frac{x-x_0}{x_{0,N}}\right)=x_{0,N}^\beta w_{\beta}\left(\frac{x-x_0}{x_{0,N}}\right)\leq u(x)+\varepsilon.
$$
The thesis \eqref{controllo1} follows letting $\varepsilon\to 0$. We have
$$
 c_sx_{0,N}^\beta :=x_{0,N}^\beta \tilde{C}\phi_1^{\beta}\left(0\right)\leq u(x_0).
$$
 Finally the complete claim of \eqref{controllo1} (namely all over the half space), follows by the scale invariance of the equation, namely by the arbitrariness of $x_0$.\\

 \noindent \textit{Next we are going to prove  \eqref{controllo2}.}\\

\noindent Let  $R_1>0$. Let $x_0=(0',t)$ be such that $t<0$ and $|t|>R_1$. Let $R_3>0$ be such that $R_3>>|t|$. Define the annular domain $\Omega_t:=B_{R_3}(x_0)\setminus \overline{B_{R_1}(x_0)}$.
 
 \noindent By \cite{CaScTr}-Theorem 1.3 and 1.5, there exists a unique radial solution $w\in W^{1,p}_{loc}(\Omega_t)$ to
 \begin{equation}\label{aux}
\begin{cases}
\displaystyle -\Delta_p w=\frac{C}{w^\gamma}& \text{in} \,\, \Omega_t\\
w>0 & \text{in} \,\,  \Omega_t\\
w=0 & \text{on} \,\, \partial \Omega_t.
\end{cases}
\end{equation}
Let $R_2>0$ such that
\begin{equation}\label{eq:funztest0}\max\left\{|t|,\frac{(R_1+3R_3)}{4}\right\}<R_2<R_3,\end{equation} and $\tilde{\Omega}_t:= (B_{R_2}(x_0)\setminus \overline{B_{R_1}(x_0)})\cap\mathbb{R}_N^+$,  see Figure \ref{fig}.
\begin{figure}[h]
\begin{center}
\includegraphics[height=.42\textwidth,width=1.1\textwidth]{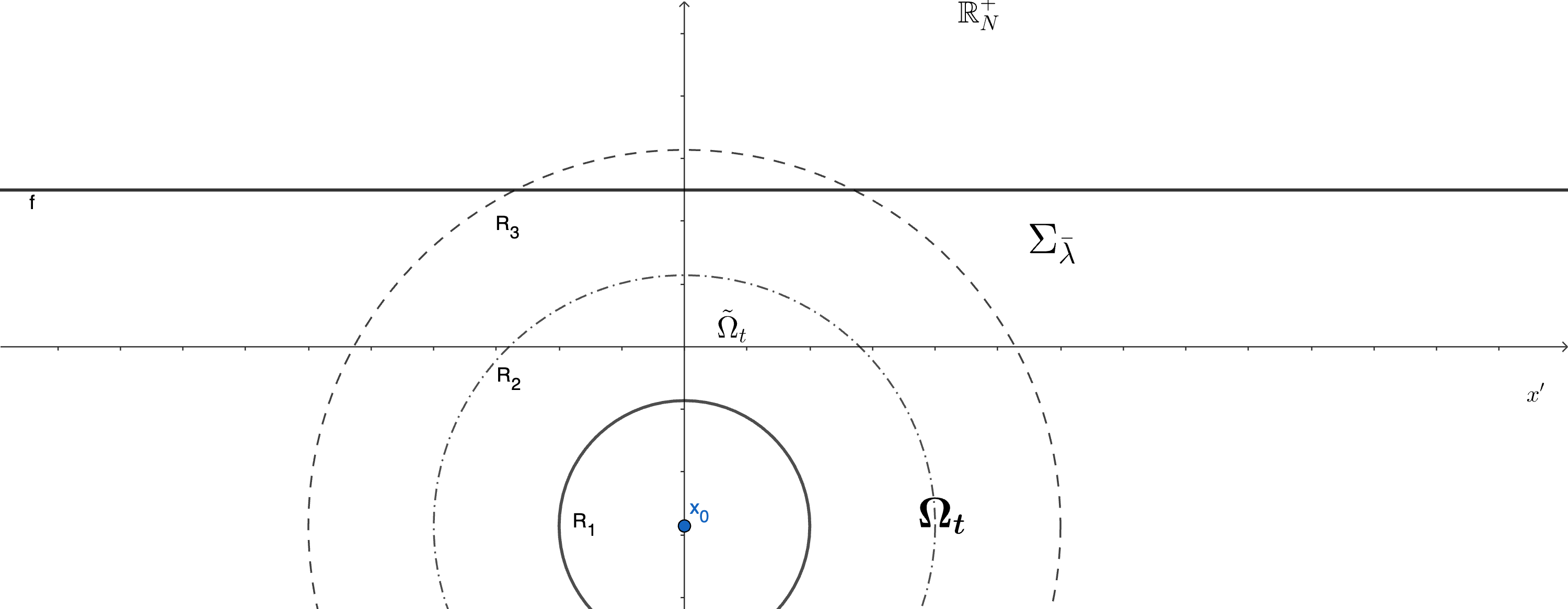}
\caption{}\label{fig}
\end{center}
\end{figure}
\noindent Denoting by $M_u$ the upper-bound of $u$ on $\Sigma_{\bar{\lambda}}$ (for a suitable $\bar{\lambda}$),  let
\begin{equation*}
k>\max\left\{ \frac{M_u}{ \displaystyle  \min_{\partial B_{R_2}(x_0)}w},1\right\}
\end{equation*}
and $w_k:=kw$. So that
\begin{equation}\label{eq:funztest1}
 w_k\geq k \min_{\partial B_{R_2}(x_0)}w>M_u\geq u, \qquad \mbox{ on }\partial B_{R_2}(x_0).
 \end{equation}
 The positivity of $w_k$ implies that $w_k>u$ on $\partial{\tilde{\Omega}_t}$.  Let us define $\varphi=(u-\varepsilon-w_k)^+\chi_{\mathbb{R}_+^N}$.

\

{\it Claim: } We show that  $\varphi\in W_0^{1,p}(\tilde{\Omega}_t)$.

\

\noindent Exploiting the fact that $w_k$ is continuous,  radial and radially decreasing in the annulus $B_{R_3}(x_0)\setminus B_{R_2}(x_0)$, thanks to  \eqref{eq:funztest0} and \eqref{eq:funztest1},  there exists an open tubular neighborhood of the sphere of radius $R_2$, say $\mathcal T_{R_2}$, such that 
\begin{equation}\label{eq:funztest2}(u-\varepsilon-w_k)^+=0 \quad\text{in}\,\,\mathcal T_{R_2}\cap {\overline{\tilde{\Omega}}_t}.\end{equation}
Let $\bar{\varphi}\in C_c^1(\mathbb{R}^N)$, $0\leq \bar{\varphi}\leq 1$ be such that $\bar{\varphi}=1$ in  ${\overline{\tilde{\Omega}}_t}\setminus \mathcal T_{R_2}$ and $\bar{\varphi}=0$ in $\mathbb{R}^N_+\setminus \tilde{\Omega}_t$. Let us denote by $\sigma_\varepsilon:=(u-\varepsilon)^+\bar{\varphi}\chi_{\mathbb{R}_+^N}$: using Dirichlet condition in \eqref{aux} (see Definition \ref{w-sol}),   $\sigma_\varepsilon\in W_0^{1,p}(\mathbb{R}_+^N)$, actually $\sigma_\varepsilon\in W_0^{1,p}(\tilde \Omega_t)$.  Let $(\phi_n)_n\subset C_c^1(\tilde \Omega_t)$ such that $\phi_n\to \sigma_\varepsilon$ in $W_0^{1,p}(\tilde \Omega_t)$. Taking a  subsequence we can suppose that $\phi_n\to (u-\varepsilon)^+\bar{\varphi}$ a.e. on $\tilde{\Omega}_t$.
 
\noindent Now let us set   $\varphi_n=(\phi_n-w_k)^+\chi_{\tilde{\Omega}_t}$. Since  $\phi_n\in C_c^1(\tilde \Omega_t)$, then $\varphi_n\in W^{1,p}_0(\tilde \Omega_t)$ and $\sup_n\|\varphi_n\|_{W^{1,p}_0(\tilde \Omega_t)}< +\infty$. Then, up to subsequences,  $\varphi_n\rightharpoonup \varphi\in W_0^{1,p}(\tilde{\Omega}_t)$ and $\varphi_n\to ((u-\varepsilon)^+\bar{\varphi}-w_k)^+$ a.e. in $\tilde{\Omega}_t$.   Hence, by \eqref{eq:funztest1},  $(u-\varepsilon-w_k)^+$ and the claim follows.

Consequently  there exists a sequence of positive functions $(\varphi_n)_{n\in\mathbb{N}}\subset C^1_c(\tilde{\Omega}_t)$ converging to $(u-\varepsilon-w_k)^+$ in $W^{1,p}_0(\tilde{\Omega}_t)$.  We set
\[
\psi_n\,:=\,\min\{(u-\varepsilon-w_k)^+,\varphi_n^+\}
\]
so that the support of $\psi_n$ is strictly contained in $\tilde{\Omega}_t$ and in the region where $u\geq w_k$. Plugging by density $\psi_n$ as test function in the problem,  we deduce that
\begin{eqnarray*}
	&&\int_{\tilde{\Omega}_t}(|\nabla u|^{p-2}\nabla u-|\nabla w_k|^{p-2}\nabla w_k,\nabla\psi_n)\\
	&\leq&C\int_{\tilde{\Omega}_t}\left(\frac{1}{u^\gamma}-\frac{1}{w_k^\gamma}\right) \psi_n.
\end{eqnarray*}

We can therefore pass to the limit and using \eqref{eq:lucio} we get
\begin{eqnarray*}
	&&C_p\int_{\tilde{\Omega}_t}(|\nabla u|+|\nabla w_k|)^{p-2}|\nabla (u-\varepsilon-w_k)^+|^2\\
	&\leq&\int_{\tilde{\Omega}_t}(|\nabla u|^{p-2}\nabla u-|\nabla w_k|^{p-2}\nabla w_k,\nabla(u-\varepsilon-w_k)^+)\\
	&=&C \int_{\tilde{\Omega}_t}\left(\frac{1}{u^\gamma}-\frac{1}{w_k^\gamma}\right) (u-\varepsilon-w_k)^+\leq 0\,,
\end{eqnarray*}
hence
\[
\int_{\tilde{\Omega}_t}(|\nabla u|+|\nabla w_k|)^{p-2}|\nabla (u-\varepsilon-w_k)^+|^2=0
\]
thus proving that $u \leq w_k+\varepsilon$ in $\tilde{\Omega}_t$. The fact that $\varepsilon$ is arbitrarily 
chosen implies that $u \leq w_k$ in $\tilde{\Omega}_t$ for any $|t|>R_1$. \\
By continuity, it follows now that
\[
u \leq w_k\quad \text{in} \,\,\tilde{\Omega}_{R_1}\,.
\]
We now consider the fact that $w_k\sim x_N^\beta$ in the annular region, as e.g. can be deduced studying the ODE radial equation basing on l'H\^opital Theorem or by \cite{LM}. This and a translation argument shows that  there exist $C_S>0$ and $\bar{\lambda}$ such that
$$
 u(x)\leq C_Sx_N^\beta.
$$
in  $\Sigma_{\bar{\lambda}}$.\\

\noindent Finally, collecting both inequalities we conclude, that each solution to \eqref{MPC}, $u$, is continuous up to the boundary.
\end{proof}
We are now ready to prove that any solution is monotone increasing near the boundary, namely:
\begin{thm}\label{monotone}
Let $u\in C^{1,\alpha}_{loc}(\mathbb R^N_+)$ be a solution to \eqref{MP}. Then there exists $\bar \lambda$ such that $u$ is monotone increasing in $\Sigma_{\bar\lambda}$
\end{thm}
\begin{proof}
By contradiction let us assume that there exists a sequence of points $(P_n)_{n\in\mathbb{N}}\subset\mathbb{R}^+_N$, $P_n:=(x_n', x_{n,N})$ such that 
\begin{equation}\label{eq:pr1}
\text{dist}(P_n, \partial \mathbb R^+_N) \rightarrow 0\qquad\text{and}\qquad\frac{\partial u}{\partial x_N}(P_n) <0. 
\end{equation}
Let us define 
\[\hat u_n(x)=\delta_n^{-\beta} u_n (\delta_n x'+x'_n,\delta_nx_N),\]
where $\delta_n=x_{n, N}$ and $\displaystyle \beta=\frac{p}{p+\gamma-1}$.
It is not difficult to see that 
\begin{equation}\begin{cases}\label{eq:pr5}
\displaystyle -\Delta_p \hat u_n=\frac{1}{\hat u_n^\gamma}+\delta_n^{\gamma\beta}f(\hat u_n) & \text{in} \,\, \mathbb{R}^N_+\\
\hat u_n>0 & \text{in} \,\,  \mathbb{R}^N_+\\
\hat u_n=0 & \text{on} \,\, \partial \mathbb{R}^N_+.
\end{cases}
\end{equation}
%Let $\varphi \in C^{\infty}_c(\mathbb R^N_+)$. We have
%\begin{eqnarray*}
%&&\int_{\mathbb R^N_+}|\nabla \hat u_n(x)|^{p-2}(\nabla \hat u_n(x), \nabla \varphi(x))\, dx\\&&=\delta_n^{\frac{(p-1)(\gamma-1)}{\gamma +p-1}}\int_{\mathbb R^N_+}|\nabla \tilde u_n(\delta_n x)|^{p-2} (\nabla \tilde u_n(\delta_nx), \nabla \varphi(x))\, dx\\
%&&=\delta_n^{\frac{(p-1)(\gamma-1)}{\gamma +p-1}+1-N}\int_{\mathbb R^N_+}|\nabla \tilde u_n(z_n)|^{p-2}(\nabla \tilde u_n(z_n), \nabla \varphi (\frac{z_n}{\delta_n}))\, dz_n\\
%&&=\delta_n^{\frac{(p-1)(\gamma-1)}{\gamma +p-1}+1-N}\left(\int_{\mathbb R^N_+}\frac{1}{\tilde u_n^\gamma(z_n)}\varphi(\frac{z_n}{\delta_n})\, dz_n+\int_{\mathbb R^N_+}f(\tilde u_n (z_n))\varphi(\frac{z_n}{\delta_n})\, dz_n\right)\\
%&&=\delta_n^{\frac{(p-1)(\gamma-1)}{\gamma +p-1}+1}\left(\int_{\mathbb R^N_+}\frac{1}{\tilde u_n^\gamma(\delta_n x)}\varphi(x)\, dx+\int_{\mathbb R^N_+}f(\tilde u_n (\delta_n x))\varphi(x)\, dx\right)\\
%&&=\int_{\mathbb R^N_+}\frac{1}{\hat u_n(x)^\gamma}\varphi(x)\, dx+\delta_n^{\frac{(p-1)(\gamma-1)}{\gamma +p-1}+1}\int_{\mathbb R^N_+}f(\tilde u_n (\delta_n x))\varphi(x)\, dx,
%\end{eqnarray*}
%namely 
%\begin{equation}\label{eq:pr5}
%-\Delta_p \hat u_n=\frac{1}{\hat u_n^\gamma}+\delta_n^{\frac{(p-1)(\gamma-1)}{\gamma +p-1}+1}f(\tilde u_n (\delta_n x)), \quad \text{in }\mathbb R^N_+.\end{equation}
For $n$ big enough,  using Theorem \ref{continuity},  we get that for all $x\in \mathbb R^N_+$
\begin{eqnarray}\label{eq:pr2}
 c_sx_N^\beta\leq \hat u_n(x)=\delta_n^{-\beta}u_n (\delta_n x'+x'_n,\delta_nx_N)\leq C_Sx_N^\beta. %\,\, \text{(for $n$ large)},\\\nonumber
%\\\nonumber
%&&\text{and}
%\\\nonumber
%\\\nonumber
%&&\hat u_n(x)\geq c_sx_N^\beta.
\end{eqnarray}
In particular  for a given  compact $K\subset \mathbb R^N_+$ we obtain that    
\begin{equation}\label{eq:pr3}
\hat u_n(x)=\delta_n^{-\beta}u_n (\delta_n x'+x'_n,\delta_nx_N)\geq c_sx^{\beta}_N\geq c_sd_K^\beta, \quad x\in K,
\end{equation}
where $d_K:=\text{dist}(K, \partial \mathbb R^N_+)$ and
\begin{equation}\label{eq:pr4}
\hat u_n(x)\leq C(p, \gamma, K), \quad x\in K.
\end{equation}
Collecting equations \eqref{eq:pr5} and \eqref{eq:pr3},   using standard regularity estimates \cite{Di}  for the $p$-laplacian operator, we deduce that 
\begin{equation}\label{eq:pr6}
\|\hat u_n\|_{C^{1,\alpha}_{loc}(\mathbb R^N_+)}\leq C(f,p, \gamma, c_s, C_S, d_K, K).
\end{equation} 
By the Arzela-Ascoli theorem, we pass to the limit in \eqref{eq:pr5}: since by {\bf (hp)}, $f$ is bounded in the set 
$[0, \|\hat u_n\|_{C^1(K)}]$, we get 
\begin{equation}\label{eq:limitprobl}
\begin{cases}
\displaystyle -\Delta_p \hat u=\frac{1}{\hat u^\gamma}& \text{in} \,\, \mathbb{R}^N_+\\
\hat u>0 & \text{in} \,\,  \mathbb{R}^N_+\\
\hat u=0 & \text{on} \,\, \partial \mathbb{R}^N_+.
\end{cases}
\end{equation}
Finally by \eqref{eq:pr2} we recover 
\begin{equation}\label{eq:prr1}
 c_sx_N^\beta\leq \hat u(x)\leq  C_Sx_N^\beta.
\end{equation}
Since \eqref{eq:limitprobl} and \eqref{eq:prr1} are in force, from Theorem 1.2 in \cite{EsSc} the  solutions to \eqref{eq:limitprobl} are  classified. In particular   
\begin{equation}\label{eq:prr2}
\hat u= C x_N^\beta,\end{equation}
for some positive constant $C=C(p,\gamma)$. On the other hand \eqref{eq:pr1} implies that  
\[\frac{\partial \hat u}{\partial x_N}(0,1)\leq 0,\] that is a contradiction with \eqref{eq:prr2}. This ends the proof.
\end{proof}

\section{Proof of Theorem \ref{mainthmczz}}\label{sect.cort}
\noindent In this section we are going to prove our main result.  \\
For a given solution $u$ to \eqref{MP},  we define the reflected function $u_\lambda(x)$ on the strip $\Sigma_{2\lambda}$ by
\begin{equation}\label{riflessa}
u_\lambda(x)=u_\lambda(x',y)\,:=\, u(x',2\lambda-y)\,\quad\text{in}\,\,\, \Sigma_{2\lambda}\,.
\end{equation}
First of all we start with the following preliminary results.

\subsection{Recovering Compacteness: the case $1<p<2$ }\label{sectioncompactbistreremdbvfj}\ \\
\noindent As customary we define  the critical set $Z_u$  by
\[
Z_u=\{x\in \mathbb{R}^N_+ \,:\, \nabla u(x)=0\}.
\]We have the following
\begin{prop}\label{trittofritto}
	Let  $1<p<2$ and let $u \in C^{1,\alpha}_{loc}({\mathbb{R}^N_+})\cap C^{0}(\overline{\mathbb{R}^N_+})$ be a solution  to \eqref{MP} under the assumptions {\bf (hp)}. \\ Suppose    that $u$ is monotone nondecreasing in $\Sigma_\lambda$, for some $ \lambda>0$.\\
	Assume that $\mathcal{U}$ is a  connected component of $\Sigma_\lambda\setminus Z_u$  such that $u(x)\equiv u_{\lambda}(x)$ in $\mathcal{U}$, (i.e. a local symmetry region for $u$).
Then   
$$ \mathcal{U}\equiv\emptyset \,.$$
\end{prop}
\begin{proof}
  \noindent Let us start observing that
  \begin{center}
     $\text{dist}(\mathcal{U},\{y=0\})>\tau>0$.
  \end{center}
By contradiction let us assume that there exists a sequence of points
\[
 x_n=(x'_n,y_n)\in \mathcal U,
\]
such that
 \begin{equation}\label{dgfygikjvb}
 \underset{n\rightarrow \infty}{\lim} \text{dist}(x_n\,,\,\{y=0\})\,=\underset{n\rightarrow \infty}{\lim} y_n=0.
\end{equation}
It follows that $u(x_n)$ approaches zero by \eqref{controllo1}-\eqref{controllo2} while $u_\lambda(x_n)$ (see \eqref{riflessa}) is bounded away from zero by the condition $u\geq C>0$,  see \eqref{controllo1}. This contradiction proves the claim.

\noindent Consequently
\begin{center}
  there exists $\sigma=\sigma(\mathcal{U},\lambda)>0$ such that $u\geq \sigma$ in $\mathcal U$.\end{center}
  Therefore,  that there exists $\sigma^+ >0$ such that
 \begin{equation}\label{fffffff}
  f(u)+\frac{1}{u^\gamma}\geq\sigma^+ \quad \text{in } \mathcal U.
 \end{equation}

We now proceed in order to conclude the proof.
Let $\varphi_R(x',y)=\varphi_R(x')$ with $\varphi_R(x') \in C^{\infty}_c (\mathbb{R}^{N-1}) $  defined as
\begin{equation}\label{Eq:Cut-off1}
 \begin{cases}
  \varphi_R \geq 0, & \text{ in } \mathbb{R}^N_+\\
   \varphi_R \equiv 1, & \text{ in } B_R^{'}(0) \subset \mathbb{R}^{N-1},\\
   \varphi_R \equiv 0, & \text{ in } \mathbb{R}^{N-1} \setminus B_{2R}^{'}(0),\\
\displaystyle |\nabla \varphi_R | \leq \frac CR, & \text{ in } (B_{2R}^{'}(0) \setminus B_R^{'}(0)) \subset  \mathbb{R}^{N-1},
\end{cases}
\end{equation}
where $B_R^{'}(0)=\left\{ x'\in \mathbb{R}^{N-1}: |x'|<R \right\}$, $ R>1$ and $C$ is a suitable positive constant.\\
For all $\varepsilon>0$, let $G_\varepsilon:\mathbb{R}^+\cup\{0\}\rightarrow\mathbb{R}$ be defined as:
	\begin{equation}\label{eq:G}
		G_\varepsilon(t)=\begin{cases} 0, & \text{if  $t\geq \varepsilon$} \\
			2t-2\varepsilon,& \text{if $t\in (\varepsilon, 2\varepsilon)$}
			\\ t, & \text{if $t\geq 2\varepsilon$}.
		\end{cases}
	\end{equation}
	\noindent Let us define
	$$
	\Psi=\Psi_{\varepsilon,R}\,:=\,\displaystyle \frac{G_\varepsilon(|\nabla u|)}{|\nabla u|}\chi_{\mathcal U}\varphi_R,
	$$
	where  $\chi_{\mathcal U}$ is the characteristic function of a set $\mathcal U$. Denoting by 	${\displaystyle h_\varepsilon (t)=\frac{G_\varepsilon(t)}{t}}$ (meaning that $h(t)=0$ for $t\in [0,\varepsilon]$),  using $\Psi$ as a test function, we get
 \begin{eqnarray}\label{VaScOoOoOoO}
   \nonumber &&\int_{\mathcal U}\varphi_{R}|\nabla u|^{p-2}(\nabla u,  \nabla h_\varepsilon(|\nabla u|)) dx
   +\int_{\mathcal U}h_\varepsilon(|\nabla u|)
    |\nabla u|^{p-2}(\nabla u, \nabla \varphi_{R})dx\\
   &\geq&\sigma^+\int_{\mathcal U} \displaystyle \varphi_Rh_\varepsilon(|\nabla u|)dx
 \end{eqnarray}
 Notice that the first integral in the left-hand-side of \eqref{VaScOoOoOoO} can be estimate as
\begin{eqnarray}\label{eq:smm3}
  &&\left| \int_{\mathcal U}\,\varphi_{R}|\nabla u|^{p-2}(\nabla u,  \nabla h_\varepsilon(|\nabla u|))dx\right|\\%\nonumber &&\leq  \int_{\mathcal U}|\nabla u|^{p-1}|h_\varepsilon'(|\nabla u|)||\nabla (|\nabla u|)|\varphi_R dx\\\nonumber
  &&\nonumber \leq \int_{\mathcal U}|\nabla u|^{p-2}\Big(|\nabla u|h_\varepsilon'(|\nabla u|)\Big)\|D^2 u\| \varphi_R dx,
\end{eqnarray}
where $\|D^2 u\|$ denotes the Hessian  norm.\\
	
	\noindent Here below, we fix $R>0$ and let $\varepsilon\rightarrow 0$. Later we will let $R\rightarrow \infty$.
	To this aim, let us first show  that
	\begin{itemize}
		\item [$(I)$] $|\nabla u|^{p-2}||D^2 u|| \varphi_R \in L^1(\mathcal U), \,\,\, \forall R>0$;
		
		\
		
		\item [$(II)$] $|\nabla u|h_\varepsilon'(|\nabla u|)\rightarrow 0$ a.e. in $\mathcal U$ as $\varepsilon \rightarrow 0$ and $|\nabla u|h_\varepsilon'(|\nabla u|)\leq C$ with $C$ not depending on $\varepsilon$.
	\end{itemize}
	Let us  prove $(i)$. Defining
	$\mathcal{D}(R)=\left\{ \mathcal{U} \cap {\{B_R^{'}(0)\times \mathbb{R}\}} \right\}$, by H\"older's inequality it follows
	\begin{eqnarray}\label{eq:smm4}
		\nonumber&&\int_{\mathcal U}|\nabla u|^{p-2}||D^2u||\varphi_R dx\leq C(\mathcal{D}(2R))\left( \int_{\mathcal{D}(2R)}|\nabla u|^{2(p-2)}||D^2u||^2\varphi_R^2 dx \right)^{\frac 12}\\
		&=&C(\mathcal{D}(2R))\left( \int_{\mathcal{D}(2R)}|\nabla u|^{p-2-\eta}||D^2u||^2\varphi_R^2|\nabla u|^{p-2+\eta}dx \right )^{\frac 12}\\\nonumber &\leq& C(\mathcal{D}(2R)) ||\nabla u||^{(p-2+\eta)/2}_{L^{\infty}(\mathbb{R}^N_+)} \left(\int_{\mathcal{D}(2R)}|\nabla u|^{p-2-\eta}||D^2 u||^2 dx \right)^{\frac 12},
	\end{eqnarray}
	where  $0\leq\eta<1$ is such that $\varphi^2_R|\nabla u|^{p-2+\eta}$ becomes bounded.
	Recall that we have proved
	$$
	\text{dist}(\mathcal{U},\{y=0\})>0.
	$$
 Using  \cite{DS}-Theorem 1.1.  we infer that
	$$\left(\int_{\mathcal{D}(2R)}|\nabla u|^{p-2-\eta}||D^2 u||^2 dx \right)^{\frac 12}\leq C.$$ Then by \eqref{eq:smm4} we obtain
	$$\int_{\mathcal U}|\nabla u|^{p-2}||D^2u||\varphi_R dx\leq C.$$
	Let us  prove $(ii)$. Recalling \eqref{eq:G}, we obtain
	$$
	h'_\varepsilon(t)=
	\begin{cases} 0 & \text{if  $t>2\varepsilon$} \\
		\frac{2\varepsilon}{t^2}& \text{if $\varepsilon< t<2\varepsilon$}
		\\ 0 & \text{if $0\leq t<\varepsilon$},
	\end{cases}
	$$
	and then $|\nabla u|h_\varepsilon'(|\nabla u|)$ tends to $0$
	almost everywhere in $\mathcal U$ as $\varepsilon$ goes to $0$; hence we have:
	$|\nabla u|h_\varepsilon'(|\nabla u|)\leq 2$.\\
	\
	
	\noindent Then by \eqref{VaScOoOoOoO}, \eqref{eq:smm3} and (I), (II) above,  passing to the limit
	as $\varepsilon \rightarrow 0$, we get:
	$$
	\int_{\mathcal U}\,
	|\nabla u|^{p-2}(\nabla u, \nabla \varphi_{R})dx
	\geq \sigma^+\int_{\mathcal U}\varphi_Rdx, \quad \forall R>0.
	$$
	Recalling \eqref{Eq:Cut-off1},
	%since  $a(\cdot)$ is locally  Lipschitz continuous,
	%$\n u \in L^{\infty}(\mathbb{R}^N_+)$ and  the solution $u$ is bounded in $\Sigma_{\lambda}$
	we have that there exists $C=C(\|\nabla u\|_{L^{\infty}(\mathbb{R}^N_+)})$ (not depending on $R$)
	such that:
	$$
	|\mathcal U \cap\{B_R'(0)\times \mathbb R\} | \leq \frac{C}{R}  |\mathcal U \cap\{B_{2R}'(0)\times \mathbb R\} |
	$$
	and $|\mathcal U \cap\{B_R'(0)\times \mathbb R\} |$ has polynomial growth. Therefore, by \cite{FMS}-Lemma 2.1
	it follows that $|\mathcal U \cap\{B_R'(0)\times \mathbb R\} |$ converges to zero if $R\rightarrow +\infty$, concluding the proof
	by contradiction.
\end{proof}

With the notations introduced at the beginning of the previous section, we set
\begin{equation}\label{MP1}
	\Lambda \,:=\,\Big\{\lambda\in\mathbb{R}^+\,\,:\,\, u\leq u_\mu \,\,\text{in}\,\, \Sigma_\mu \,\,\,\forall \mu<\lambda\Big\}\,
\end{equation}
and we define
\begin{equation}\label{barlambda}
	\bar \lambda \,:=\, \sup\,\Lambda\,.
\end{equation}
This value is well defined recalling Theorem \ref{monotone}.
%It follows easy by continuity that
%\[
%u\leq u_{\lambda}\qquad\text{in}\,\,\, \Sigma_{\bar\lambda}\,.
%\]
We have the following:

\begin{prop}\label{lellalemma} Let  $1<p<2$ and let $u \in C^{1,\alpha}_{loc}({\mathbb{R}^N_+})\cap C^{0}(\overline{\mathbb{R}^N_+})$ be a solution  to \eqref{MP} under the assumptions {\bf (hp)}.  
	
	Assume $ 0 < \bar\lambda < + \infty$  and set $$W_\varepsilon:=\Big(u-u_{\bar \lambda+\varepsilon}\Big)\cdot \chi_{\{y\leqslant \bar{\lambda}+\vep\}}$$
	where $ \varepsilon >0$.

	Given $0<\delta <\frac{\blambda}{2}$ and $\rho>0$, there exists $\vep_0 >0 $ such that, for any $\vep\leqslant \vep_0$, it follows
	\[
	\text{Supp} \,W_\vep^+\subset \{\bar{\lambda}-\delta\leqslant y\leqslant\bar{\lambda}+\vep\}\cup \left(\bigcup_{x'\in\mathbb R^{N-1}} B_{x'}^\rho\right),
	\]
	where $B_{x'}^\rho$ is such that
	\begin{equation}\label{njgkdngkfkkjc}
		B_{x'}^\rho\subseteq\left\{y\in(0,\bar\lambda+\varepsilon)\,:\, |\nabla u(x',y)|<\rho,\,|\nabla u_{\bar \lambda +\varepsilon}(x',y)|<\rho\right\}.
	\end{equation}
\end{prop}

\begin{proof}
	
	Assume by contradiction that there exists $\delta >0$, with $\displaystyle 0<\delta <\frac{\blambda}{2}$,
	such that, given any $\vep_0>0$, we find $\vep\leqslant \vep_0$ and
	$x_\vep=(x'_\vep,y_\vep)$ such that:
	\begin{itemize}
		\item [$(i)$]  $u(x'_\vep,y_\vep)\geqslant u_{\bar{\lambda}+\vep}(x'_\vep,y_\vep)$;
		\item  [$(ii)$] $x_\varepsilon$ belongs to the set
		$$
		\Big\{(x',y) \in\mathbb R^N :  y_\vep\leqslant \bar{\lambda}-\delta\Big\}, $$
		and it holds the alternative: 
\begin{center}
either $|\nabla u (x_\vep)|\geq\rho$ or $|\nabla u_{\bar \lambda +\varepsilon}(x_\varepsilon)|\geq\rho
		$.
\end{center}
	\end{itemize}
	Take now $\displaystyle \vep_0=\frac{1}{n}$, then there exists $\displaystyle \vep_n \leqslant \frac{1}{n}$  and
	a sequence $$
	x_n=(x'_{n},y_{n})=(x'_{\vep_n},y_{\vep_n}),
	$$
	such that
	$$
	u(x'_n,y_n)\geqslant u_{\bar{\lambda}+\vep_n}(x'_n,y_n),
	$$
	and satisfying $(ii)$ above. Up to subsequences we may assume that:
	$$
	y_n\rightarrow y_0 \quad \text{as} \quad n\rightarrow +\infty,\qquad
	\text{ with } \quad0<\underline \delta\leqslant y_0\leqslant \bar{\lambda}-\delta\,, 
	$$
	where the bound $ y_0 \geqslant \underline \delta>0 $ follows by \eqref{controllo1} and \eqref{controllo2} and \eqref{riflessa}.
	Let us now define
	\begin{equation}\label{cazzo}
		\tu_n(x',y)=u(x'+x'_n,y),
	\end{equation}
	satisfying 
 \begin{equation}\label{pun}
		\int_{\mathbb{R}^N_+}|\nabla \tilde u_n|^{p-2}(\nabla \tilde u_n,\nabla \varphi)dx=\int_{\mathbb{R}^N_+}\, \left ( \frac{1}{\tilde u_n^\gamma}+f(\tilde u_n) \right)\varphi dx \quad \forall \varphi \in C^{\infty}_c(\mathbb{R}^N_+).
	\end{equation}
	Since  $u$  is  bounded on every strip $\Sigma_\eta$, $ \eta>0$, as before,  by DiBenedetto-$C^{1,\alpha}$ estimates applied to the problem satisfied by $u_n$,  Ascoli's Theorem and a standard diagonal process we get that :
	\begin{equation}\label{roccaseccatrubis}
		\tu_n\overset{C^{1,\alpha'}_{loc}({{\mathbb{R}^N_+}})} {\longrightarrow}\tu,
	\end{equation}
	(up to subsequences) for $\alpha '<\alpha$.
	\
	
	\noindent \emph{We claim that}
	\begin{itemize}
		\item[-] $\tu\geqslant 0 $ in $\mathbb{R}^N_+$, with $\tu(x,0)=0$ for every $x \in \mathbb{R}^{N-1}$;
		\item [-] $\tu\leqslant \tu_{\blambda}$ in $\Sigma_{\blambda};$
		\item[-] $\tu(0,y_0)= \tu_{\blambda}(0,y_0)$;
		%$\tu(0,y_0)\geqslant \tu_{\blambda}(0,y_0)$ (with actually $\tu(0,y_0)= \tu_{\blambda}(0,y_0)$);
		\item [-] $|\nabla \tu (0,y_0)|\geq \rho$.
	\end{itemize}\ \\
	
	\noindent Indeed $\tu\geqslant 0 $ in $\mathbb{R}^N_+$ by construction and $\tu(x,0)=0$ or every $x \in \mathbb{R}^{N-1}$ thanks to the bound in \eqref{controllo1} and \eqref{controllo2}.
	Moreover by continuity $\tu\leqslant \tu_{\blambda}$ in $\Sigma_{\blambda}$ and $\tu(0,y_0)\geqslant \tu_{\blambda}(0,y_0)$.
	Actually there holds: $\tu(0,y_0)= \tu_{\blambda}(0,y_0)$.\\
	Finally, at $x_0=(0,y_0)$
	(where $\tu(0,y_0)= \tu_{\blambda}(0,y_0)$)
	we have that $\nabla \tilde u(0,y_0)=\nabla\tilde u_{\bar \lambda}(0,y_0)$,
	because $x_0$ is an interior minimum point for the function
	$w(x):=\tilde u_{\blambda}(x)-\tilde u(x)\geq 0$.
	For all $n$ we have
	$|\nabla u (x_n)|\geq\rho$ or $|\nabla u_{\bar \lambda +\varepsilon_{n}}(x_n)|\geq\rho$,
	and, using the uniform~$C^1$ convergence on compact set,  we get: $|\nabla \tu (0,y_0)|\geq \rho$.\\
	Passing to the limit in \eqref{pun} we obtain that 
	\begin{equation}\nonumber
		\int_{\mathbb{R}^N_+}|\nabla \tilde{u}|^{p-2}(\nabla \tilde{u},\nabla \varphi)dx=\int_{\mathbb{R}^N_+}\, \left ( \frac{1}{\tilde{u}^\gamma}+f(\tilde{u}) \right)\varphi dx \quad \forall \varphi \in C^{\infty}_c(\mathbb{R}^N_+).
	\end{equation}
	\noindent Note that, furthermore, $\tu > 0$ in $\mathbb{R}^N_+$ thanks to  \eqref{controllo1}.
	  Moreover we have  $\tu\leqslant \tu_{\blambda}$ in $\Sigma_{\blambda}$
	and   $\tu(0,y_0)= \tu_{\blambda}(0,y_0)$ so that,  by the strong comparison principle \cite{PS},
	  we deduce that
	$\tu= \tu_{\blambda}$ in the connected component  $\mathcal{U}$, of $\Sigma_{\bar\lambda}\setminus Z_u$, containing the point $(0,y_0)$, where $|\nabla \tu (0,y_0)|\geq \rho$. This provides a contradiction by  Proposition \ref{trittofritto}  concluding the proof.
\end{proof}

\subsection{Recovering Compacteness: the case $p>2$ }

\begin{prop}\label{lellalemmacaz} Let  $p\geq 2$ and let $u \in C^{1,\alpha}_{loc}({\mathbb{R}^N_+})\cap C^{0}(\overline{\mathbb{R}^N_+})$ be a solution  to \eqref{MP} under the assumptions {\bf (hp)}.  
	
	Assume $ 0 < \bar\lambda < + \infty$  and set $$W_\varepsilon:=\Big(u-u_{\bar \lambda+\varepsilon}\Big)\cdot \chi_{\{y\leqslant \bar{\lambda}+\vep\}}$$
	where $ \varepsilon >0$.

	Given $0<\delta <\frac{\blambda}{2}$,  there exists $\vep_0 >0 $ such that, for any $\vep\leqslant \vep_0$, it follows
	\[
	\text{supp} \,W_\vep^+\subset \{\bar{\lambda}-\delta\leqslant y\leqslant\bar{\lambda}+\vep\}.
	\]
\end{prop}

\begin{proof}
	Exactly as in the proof of Proposition \ref{lellalemma} we 
	assume by contradiction that there exists $\delta >0$, with $\displaystyle 0<\delta <\frac{\blambda}{2}$,
	such that, given any $\vep_0>0$, we find $\vep\leqslant \vep_0$ and
	$x_\vep=(x'_\vep,y_\vep)$ such that:
	\begin{itemize}
		\item [$(i)$]  $u(x'_\vep,y_\vep)\geqslant u_{\bar{\lambda}+\vep}(x'_\vep,y_\vep)$;
		\item  [$(ii)$] $x_\varepsilon$ belongs to the set $\Sigma_{\bar{\lambda}-\delta}$.
%		$$
%		\Big\{(x',y) \in\mathbb R^N :  y\leqslant \bar{\lambda}-\delta\Big\} $$
	\end{itemize}
	Take now $\displaystyle \vep_0=\frac{1}{n}$, then there exists $\displaystyle \vep_n \leqslant \frac{1}{n}$  and
	a sequence $$
	x_n=(x'_{n},y_{n}):=(x'_{\vep_n},y_{\vep_n}),
	$$
	such that
	$$
	u(x'_n,y_n)\geqslant u_{\bar{\lambda}+\vep_n}(x'_n,y_n),
	$$
	and satisfying conditions $(ii)$ above. Up to subsequences we may assume that:
	$$
	y_n\rightarrow y_0, \quad \text{as} \quad n\rightarrow +\infty,\qquad
	\text{ with } \quad0<\underline \delta\leqslant y_0\leqslant \bar{\lambda}-\delta\,, 
	$$
	where the bound $ y_0 \geqslant \underline \delta>0 $ follows as in the proof of Proposition \ref{trittofritto}. 
	Let us now define
	\begin{equation}\label{cazzocaz}
		\tu_n(x',y)=u(x'+x'_n,y).
	\end{equation}
	
	Since both $u$  is  bounded on every strip $\Sigma_\eta$, $ \eta>0$, as before, by DiBenedetto-$C^{1,\alpha}$ estimates, Ascoli's Theorem and a standard diagonal process we get that:
	\begin{equation}\label{roccaseccatrubiscaz}
		\tu_n\overset{C^{1,\alpha'}_{loc}({ {\mathbb{R}^N_+}})} {\longrightarrow}\tu,
	\end{equation}
	(up to subsequences) for $\alpha '<\alpha$.\ \\
By continuity $\tu\leqslant \tu_{\blambda}$ in $\Sigma_{\blambda}$ and by $(ii)$, $\tu(0,y_0)\geqslant \tu_{\blambda}(0,y_0)$.\\
Taking into account that
 \begin{equation*}
		\int_{\mathbb{R}^N_+}|\nabla u_n|^{p-2}(\nabla u_n,\nabla \varphi)dx=\int_{\mathbb{R}^N_+}\, \left ( \frac{1}{u_n^\gamma}+f(u_n) \right)\varphi dx \quad \forall \varphi \in C^{\infty}_c(\mathbb{R}^N_+),
	\end{equation*}
        passing to the limit we obtain that 
	\begin{equation}\nonumber
		\int_{\mathbb{R}^N_+}|\nabla \tilde{u}|^{p-2}(\nabla \tilde{u},\nabla \varphi)dx=\int_{\mathbb{R}^N_+}\, \left ( \frac{1}{\tilde{u}^\gamma}+f(\tilde{u}) \right)\varphi dx \quad \forall \varphi \in C^{\infty}_c(\mathbb{R}^N_+).
	\end{equation}
	\noindent By construction  $\tu \geqslant 0$ in $\mathbb{R}^N_+$; indeed by \eqref{controllo1} we deduce that  $\tu >0$. 
%by the strong maximum principle \cite{DS2006}-Theorem 1.4,
	Since $\tu(0,y_0)= \tu_{\blambda}(0,y_0)$ with $ y_0 \geqslant \underline \delta>0$, we are far from $\partial \mathbb{R}^N_+$ and the right term of \eqref{MP} is Lipschitz.  Moreover $\tu\leqslant \tu_{\blambda}$ in $\Sigma_{\blambda}$ so that applying here the  strong comparison principle \cite{DS2006}-Theorem 1.4, we deduce that
	$\tu= \tu_{\blambda}$ in  $\Sigma_{\bar\lambda}$,  concluding the proof by contradiction thanks to the Dirichlet condition.
\end{proof}

The proof is based on the moving planes procedure.  By Theorem \ref{monotone} the set $\Lambda$ defined in \eqref{MP} is not empty and $ \bar\lambda \in (0, +\infty]$. To conclude the proof we need to show that $\bar\lambda=\infty$.\\
\begin{proof}[Proof of Theorem \ref{mainthmczz}]

Assume that $ \bar\lambda$ is finite.  %and set $ \lambda_0= {\bar\lambda}+2 $ and 
% and consider $\e_0>0$ fixed. We shall prove the thesis exploiting Theorem \ref{th:wcpstrip}
%in the strip $\Sigma_{\bar\lambda+2\e_0}$, namely $y_0=\lambda_0=\frac{\bar\lambda}{2}+\e_0$ in %\eqref{eq:gadkgdfakhsjgfasfgjhksagfjfgjasgjh}. In this case we can choose
%\[
%M_0\,:=  \Vert u\Vert_{L^\infty(\{ 0 \le y \le {2\bar\lambda}+10 \})} + \Vert \nabla u \Vert_{L^\infty(\{ 0 \le y \le {2\bar\lambda}+10 \})} +1  >0
%\]

%\noindent and take $\tau_0=\tau_0(N,p,q,\lambda_0, M_0, \gamma)>0$
%and $\eps_0=\eps_0(N,p,q,\lambda_0, M_0, \gamma)>0$ as in CHIARIRE XXXXXXXX.

By Proposition  \ref{lellalemma} we have that, given some small $\delta,\rho>0$, %$0<\delta < \min \{\frac{\blambda}{2}, \frac{\tau_0}{4}\}$ and
%$0< \rho < \eps_0$,  
 for any $\vep>0$ sufficiently small %slant \bar\e$ %\min \{, \frac{\tau_0}{4},1\}$, 
it follows

\[
\text{Supp} \,W_\vep^+\subset \{\bar{\lambda}-\delta\leqslant y\leqslant\bar{\lambda}+\vep\}\cup \left(\bigcup_{x'\in\mathbb R^{N-1}} B_{x'}^\rho\right)\quad \text{for}\,\,1<p<2,
\]
while 

\[
\text{Supp} \,W_\vep^+\subset \{\bar{\lambda}-\delta\leqslant y\leqslant\bar{\lambda}+\vep\}\quad \text{for}\,\,p\geq 2,
\]
as in Proposition~\ref{lellalemmacaz}, where $W_\vep^+ =(u-u_{\bar{\lambda}+\vep})^+\chi_{\{y\leqslant
	\bar{\lambda}+\vep\}}$ and $B_{x'}^\rho$ is defined in \eqref{njgkdngkfkkjc}. \ \\

Applying \cite{FMRS}-Theorem 1.7 when  $1<p<2$ and \cite{FMS PISA}-Proposition~3.3 for $p>2$, we deduce that $ u \le u_{\bar{\lambda}+\vep} $ in $\Sigma_{\bar\lambda+ \vep}$, which contradicts the definition of $\blambda$ and yields that $\blambda=\infty$. This, in turn,  implies the desired monotonicity of $u$, that is $\displaystyle  \frac{\partial u}{\partial x_N}(x',x_N)\geqslant 0$ in  $\mathbb{R}^N_+ $.

\end{proof}

\begin{center}
{\bf Acknowledgements}
\end{center} 
L. Montoro, L. Muglia, B. Sciunzi  are partially supported by PRIN project P2022YFAJH 003 (Italy): Linear and nonlinear PDEs; new directions
and applications.\ \\ \ \\
%
%\
%
\noindent All the authors are partially supported also by Gruppo Nazionale per l' Analisi Matematica, la Probabilit\`a  e le loro Applicazioni (GNAMPA) of the Istituto Nazionale di Alta Matematica (INdAM).

\vspace{1cm}

%\begin{center}{\bf Acknowledgements}\end{center}  
%L. Montoro and B. Sciunzi are partially supported by PRIN project 2017JPCAPN (Italy): Qualitative and quantitative aspects of nonlinear PDEs, and L. Montoro by  Agencia Estatal de Investigaci\'on (Spain), project PDI2019-110712GB-100.

\

\begin{center}
{\sc Data availability statement}

\

All data generated or analyzed during this study are included in this published article.
\end{center}

\

\begin{center}
	{\sc Conflict of interest statement}
	
	\
	
	The authors declare that they have no competing interest.
\end{center}

\bibliographystyle{elsarticle-harv}

\end{document}